\documentclass[12pt]{article}
\usepackage{amsmath,amssymb,amsthm,amscd}
\title{Remarks on hard Lefschetz conjectures on Chow groups}
\author{Baohua Fu}
\newtheorem{Thm}{Theorem}
\newtheorem{Lem}{Lemma}
\newtheorem{Prop}{Proposition}
\newtheorem{Cor}{Corollary}
\newtheorem{Conj}{Conjecture}
\newtheorem{Rque}{Remark}

\def\qit{{\mathbb Q}}
\def\zit{{\mathbb Z}}
\def\pit{{\mathbb P}}

\def\0{{\mathcal O}}
\def\min{\mathop{\rm min}\nolimits}

\def\k{{\mathbf k}}

\begin{document}
\maketitle
\begin{abstract}
We propose two conjectures of Hard Lefschetz type on Chow groups and prove them
for some special cases. For abelian varieties, we shall show they are equivalent
to well-known conjectures of Beauville and Murre.
\end{abstract}
\section{Introduction}
Let $X$ be a smooth complex projective variety of dimension $n$
and $D$ an ample divisor on $X$. We denote by $CH^p(X)$ the Chow
group of codimension $p$ algebraic cycles on $X$ with {\em
rational coefficients}. Let $cl: CH^p(X) \to H^{2p}(X, \qit)$ be the class map. We have the following commutative
diagram:
$$
\begin{CD}
CH^p(X) @> \cdot D^{n-2p}>>  CH^{n-p}(X) \\
@VVclV   @VVclV \\
 H^{2p}(X, \qit) @>\cup cl(D)^{n-2p}>> H^{2n-2p}(X, \qit)
\end{CD}
$$

The hard Lefschetz theorem on cohomology  says that the bottom map $\cup cl(D)^{n-2p}$
is an isomorphism. Note that the map $\cdot D^{n-2p}$ is in general not surjective, since
by a famous result of Mumford, $CH^n(X)$ could be very complicated while $CH^0(X) \simeq \qit$.
We propose in this paper the following
conjecture on Chow groups.
\begin{Conj} \label{Conj1}
If $n \geq 2p$, the map $CH^p(X)\xrightarrow{\cdot D^{n-2p}}
CH^{n-p}(X)$ is injective.
\end{Conj}

For an abelian variety defined over a finite field, it is proven
by Soul\'e (\cite{S}) that the above map is an isomorphism.
Another conjecture of Hard Lefschetz type that we will discuss is
the following:
\begin{Conj} \label{Conj2}
If $n \geq 2p-1$, the map $CH^p(X)_{hom}\xrightarrow{\cdot
D^{n-2p+1}} CH^{n-p+1}(X)_{hom}$ is injective, where
$CH^p(X)_{hom} = ker(CH^p(X) \xrightarrow{cl} H^{2p}(X)).$
\end{Conj}

{\em Convention:} We will say Conjecture 1 or 2 for a pair $(X,
D)$ or a triple $(X, D, p)$ when we want to emphasize the special
$D$ or $p$.

As easily seen, Conjecture \ref{Conj2} implies Conjecture
\ref{Conj1}. We shall prove in the next section that both
conjectures follow from the  standard conjecture of Lefschetz type (\cite{G}) and
Bloch-Beilinson's conjecture (see for example \cite{J}) on the
existence of a functorial filtration on Chow groups.

In the third section, we prove both conjectures for $p=1$, which
shows that Conjecture \ref{Conj1} (resp. Conjecture \ref{Conj2})
holds when $\dim(X) \leq 4$ (resp. $\dim(X) \leq 3$). With the
help of \cite{B-S},  some partial results for $p=2$ are obtained
in this section, which shows that Conjecture \ref{Conj1} (resp.
Conjecture \ref{Conj2}) holds for rationally connected manifolds
of dimension $\leq 6$ (resp. $\leq 5$).

In section 4, we first prove that if Conjecture \ref{Conj1} (resp.
Conjecture \ref{Conj2}) holds for $(X, D)$, then it also holds for
$(Z, f^*(D)+m Exc(f))$, where $f: Z \to X$ is the blow-up of $X$
 along a
smooth subvariety of dimension $\leq 2$ (resp. $\leq 1$)
and $m<0$ is a rational number such that
$f^*(D)+m Exc(f)$ is ample. Then we
prove that if Conjecture \ref{Conj1} (resp. Conjecture
\ref{Conj2}) holds for $X$, then it holds for $X \times
\pit^m$. Also we prove Conjecture \ref{Conj1} (resp. Conjecture
\ref{Conj2}) holds for projective bundles over smooth varieties of
dimension $\leq 2$ (resp. $\leq 1$). Together with the examples of
small dimension, these two properties
provide many examples supporting Conjecture \ref{Conj1} and Conjecture
\ref{Conj2}.

In the fifth section, we reveal the relationship between
Conjecture \ref{Conj1} and Beauville's conjecture (\cite{B1},
\cite{B2}). More precisely, we prove that Beauville's conjecture
is equivalent to Conjecture \ref{Conj1} for abelian varieties or
to Conjecture \ref{Conj1} for symmetric products $(C^{(k)},z_k)$
of curves, where $z_k$ is the ample divisor $C^{(k-1)}+pt$.
 Finally we prove in the last
section that for abelian varieties, Murre's conjectures (see
\cite{Mu}) are equivalent to Conjecture \ref{Conj2}. In
particular, this proves Conjecture \ref{Conj1} (resp. Conjecture
\ref{Conj2})
for $(C^{(k)}, z_k)$ with $g(C) \leq 4$ (resp. $g(C) \leq 3$). \\

{\em Acknowledgements:} It is my pleasure to thank A. Beauville, A. Collino and
C. Voisin for their suggestions and corrections to a previous version.
Part of this work has been done during my visits to FuDan University and to KIAS.
I'd like to thank M. Chen and J.-M. Hwang for the invitations and for
the helpful discussions.

\section{Bloch-Beilinson's conjecture}

Recall that the  Bloch-Beilinson conjecture (see for example
\cite{J}) asserts the existence of  a decreasing filtration $F^i
CH^p(X)$ for any smooth projective variety $X$, satisfying the
following properties:

(i) $F^0CH^p(X)= CH^p(X)$, $F^1 CH^p(X) = CH^p(X)_{hom}$;

(ii) $F^\cdot$ is stable under correspondences;

(iii) If $\Gamma \in CH^k(X \times Y)$ induces the zero map from
$H^{2l-i}(X, \mathbb{Q})$ to $H^{2k+2l-2n-i}(Y, \mathbb{Q}))$,
then so is the map $Gr_{F^\cdot}^i \Gamma: Gr_{F^\cdot}^i CH^l(X)
\to Gr_{F^\cdot}^i CH^{k+l-n}(Y)$, where $n=dim(X)$;

(iv) $F^{p+1} CH^p(X) = 0$.

\begin{Prop}
Assume the Bloch-Beilinson conjecture and the standard
conjecture of Lefschetz type(\cite{G}), then Conjecture \ref{Conj1} and Conjecture \ref{Conj2}
hold.
\end{Prop}
\begin{proof}
The proof here is inspired by \cite{V2}. Up to replacing $D$ by a
higher multiple, we may assume that $D$ is a smooth irreducible
divisor on $X$ such that $Y:=D^{n-i}$ is again smooth and
irreducible. Let $\iota: Y \to X$ be the natural inclusion and
$\Gamma_\iota$ its graph. Then the map $ \cdot D^{n-i}$ is induced
by the cycle $\Gamma:= \Gamma_\iota \circ \Gamma_\iota^ t \in
CH^{2n-i}(X \times X)$. By the  standard conjecture of Lefschetz type,
there exists an algebraic cycle $\Gamma' \in CH^i(X \times X)$
such that $\Gamma'_* \circ \Gamma_* = Id: H^i(X, \mathbb{Q}) \to
H^i(X, \mathbb{Q}).$

By Bloch-Beilinson conjecture (iii), we have $\Gamma'_* \circ
\Gamma_* = Id:  Gr_{F^\cdot}^v CH^p(X) \to Gr_{F^\cdot}^v
CH^{p}(X)$ for $2p-v=i.$ In particular, $\Gamma_* = \cdot D^{n-i}:
Gr_{F^\cdot}^v CH^p(X) \to Gr_{F^\cdot}^v CH^{p+n-i}(X)$ is
injective for $2p-v=i.$ This implies that the map $$
Gr_{F^\cdot}^v CH^p(X) \xrightarrow{\cdot D^{n-2p}} Gr_{F^\cdot}^v
CH^{n-p}(X)$$ is injective for all $v \geq 0$. By statement (iv)
in Bloch-Beilinson conjecture, Conjecture \ref{Conj1} holds.

For Conjecture \ref{Conj2}, the same argument shows that $
Gr_{F^\cdot}^v CH^p(X) \xrightarrow{\cdot D^{n-2p+1}}
Gr_{F^\cdot}^v CH^{n-p+1}(X)$ is injective for all $v \geq 1$. Now
the conclusion follows from statement (i) in Bloch-Beilinson
conjecture.
\end{proof}

As is well-known, the Bloch-Beilinson conjecture is fundamental to
understand Chow groups, while it is extremely hard to prove.
Hopefully, the two conjectures of Hard Lefschetz type could be
easier to be checked for various examples.

\section{Small codimension}

The following Lemma is immediate from the Hard Lefschetz theorem
on cohomologies.
\begin{Lem}\label{2Imply1}
Conjecture \ref{Conj2} for $(X,D,p)$  implies Conjecture \ref{Conj1} for $(X,D,p)$.
\end{Lem}

We prove in the following that Conjecture \ref{Conj1} and
\ref{Conj2} are true for $p=1$, which could be already known.
This implies that Conjecture \ref{Conj1} holds when $\dim(X) \leq
4$ and Conjecture \ref{Conj2} holds when $\dim(X) \leq 3$.
\begin{Prop}\label{p=1}
Let $X$ be a smooth projective variety of dimension $n$ and $D$ an ample divisor.
Then $CH^1(X) \xrightarrow{\cdot D^{n-2}} CH_1(X)$ and
$CH^1(X)_{hom} \xrightarrow{\cdot D^{n-1}} CH_0(X)_{hom}$ are
injective.
\end{Prop}
\begin{proof}
Consider the Albanese map $CH_0(X, \mathbb{Z})_{hom}
\xrightarrow{alb} Alb(X)$ and the composition $alb \circ (D^{n-1}
\cdot): Pic^0(X) \to Alb(X)$. By the Hard Lefschetz theorem, the map $H^1(X,
\mathbb{Q}) \xrightarrow{\cup cl(D)^{n-1}} H^ {2n-1}(X,
\mathbb{Q})$ is an isomorphism, which implies that the map $alb \circ
(D^{n-1} \cdot)$ is an isogeny. In particular, it is injective
after tensor product with $\mathbb{Q}$, which gives the second
statement. The first one follows from Lemma \ref{2Imply1}.
\end{proof}

Here are some consequences of this.
\begin{Cor} \label{p}
Let $C$ be a smooth projective curve of genus $g$ and $J$ its
Jacobian. Fix one point on $C$ and consider the associated
Abel-Jacobi map: $C \xrightarrow{\iota} J$. Then the induced map
$\iota_*: CH_0(C)_{hom} \to CH_0(J)_{hom}$ is injective.
\end{Cor}
\begin{proof}
Note that $\iota^*: Pic^0(J) \to Pic^0(C)$ is surjective. Let
$\alpha \in  CH_0(C)_{hom}$ with $\iota_*(\alpha)=0$. Then there
exists an element $\beta \in CH^1(J)_{hom}$ such that
$\alpha=\iota^* \beta$, then $0=\iota_*(\alpha)=\iota_*(\iota^*
\beta)= \beta \cdot \iota_*(C).$ By \cite{B2}, we have a
decomposition $\iota_*(C) = \sum_{s=0}^{g-1} C_{(s)}$, where
$C_{(s)} \in CH^{g-1}_{(s)}(J)$. As $\beta$ is in $CH^1_{(1)}(J)$,
we obtain that $\beta \cdot C_{(0)} = 0$.  Let $\Theta$ be a
symmetric theta divisor. We have
 $C_{(0)} = \frac{\Theta^{g-1}}{(g-1)!}$. As a consequence,  $\beta \cdot
\Theta^{g-1}=0$, which gives $\beta=0$ by the previous
Proposition, so $\alpha=0$.
\end{proof}
\begin{Rque}
As pointed out by C. Voisin, one can prove this corollary by noting
that $CH_0(C)_{hom} \to Alb(C)$ is injective and
$Alb(C)=Alb(J(C))$.
\end{Rque}

\begin{Cor}
Let $\iota: C \to S$ be the inclusion of an ample divisor in a
smooth projective surface $S$ such that $q(S)=g(C)$. Then the map
$ \iota_*: CH_0(C)_{hom} \to CH_0(C)_{hom}$ is injective.
\end{Cor}
\begin{proof}
Consider the map $\iota^*: Pic^0(S) \to Pic^0(C)$.  If
$\iota^*(\alpha)=0$ for some $\alpha \in Pic^0(S)$, then $\alpha
\cdot C=0$, which implies by the previous proposition that
$m\alpha=0$ for some $m>0$, i.e. $\alpha$ is a torsion point. The
condition $q(S)=g(C)$ implies then $\iota^*$ is an isogeny, which
is in particular surjective. Now a similar argument as in the
proof of the previous Lemma concludes the proof.
\end{proof}
\begin{Rque}\upshape
(i). In \cite{V}, a similar but much more difficult
 result is proved for general hyperplane
sections of a degree $d \geq 5$ surface $S$ in $\mathbb{P}^3$.
However, our result cannot be applied here, since the condition
$q(S)=g(C)$ is never satisfied in this case.

(ii). Examples of  pairs $(S, C)$
such that $C$ is ample and $g(C)=q(S)$
include $(\mathbb{P}^1-bundle, section)$,
$(C^{(2)}, C+pt)$, $(C_1 \times C_2, C_1+ C_2)$ and so on (for a classification of
such pairs see \cite{F}).
\end{Rque}

Now we study the case of $p=2$, where the main difficulty is that
$CH^2(X)$ is in general ``infinite-dimensional".
\begin{Prop}
Let $X^n$ be a smooth projective variety and $V \subset X$ a
closed subvariety of dimension $\leq 1$ such that $CH_0(V) \to
CH_0(X)$ is surjective. Then Conjecture \ref{Conj1} and Conjecture
\ref{Conj2} hold for $CH^2(X)$.
\end{Prop}
\begin{proof}
Let $J^2(X)$ be the second intermediate Jacobian of $X$. By
\cite{B-S}, the hypothesis implies that $CH^2(X)_{hom} \to J^2(X)
\otimes \qit$ is injective. By the Hard Lefschetz theorem, $\cup
cl(D^{n-3}): H^3(X, \qit) \to H^{2n-3}(X, \qit)$ is an isomorphism
for any ample divisor $D$ on $X$, which implies that  the map
$\cdot D^{n-3}: J^2(X) \to J^{n-1}(X)$ has finite kernel.  As in the
proof of Proposition \ref{p=1}, this implies Conjecture
\ref{Conj2}.
\end{proof}

Note that $CH_0$ is birationally invariant, and $CH_0(F)=\qit$ for any
rationally connected manifold $F$. The previous proposition gives the following result.
\begin{Cor}
Conjecture \ref{Conj1} (resp. Conjecture \ref{Conj2}) holds for
smooth projective varieties of dimension $\leq
6$ (resp. $\leq 5$) which are
fibrations over a smooth curve with general fibers being
rationally connected.
\end{Cor}

\section{Blow-ups and projective bundles}
Now we study the behavior of both conjectures under blow-ups. Let
$\iota: Y \to X$  be the inclusion of a smooth subvariety $Y$ of
codimension $r+1 \geq 2$ in a smooth projective variety $X^n$.
Write $d:=\dim Y = n-r-1$. We
denote by $f: Z \to X$ the blow-up of $X$ along $Y$ and $E$ the
exceptional divisor. We use the following notations for morphisms:
$$
\begin{CD}
E @>>j> Z \\
@VVgV   @VVfV \\
Y @>>\iota> X
\end{CD}
$$
\begin{Prop} \label{blowup}
Assume Conjecture \ref{Conj1} (resp. Conjecture \ref{Conj2}) holds
for $(X, L)$ with $L$ an ample divisor on $X$, then Conjecture
\ref{Conj1} (resp. Conjecture \ref{Conj2}) holds for $(Z, f^*(L)+m
E, p)$ with $p \geq d$ (resp. $p \geq d-1$), where $m<0$
is a rational number such that $f^*(L)+mE$ is ample on $Z$.
\end{Prop}
\begin{proof}
We will give details for the case of Conjecture \ref{Conj1}, while
the proof of the case of  Conjecture \ref{Conj2} is completely
similar. As the case $p=1$ has been proved in the last section, we
may assume that $p \geq 2$.

Take an element $\alpha \in CH^p(Z)$. By the blow-up formula for Chow
groups, we can write $\alpha = f^*(x) + \sum_{i=0}^r j_*(h^{p-1-i}
g^* y_i)$, where $x \in CH^p(X)$, $y_i \in CH^{i}(Y)$ and $h
= c_1(\mathcal{O}_g(1)).$ Assume now $(f^*L+mE)^{n-2p} \alpha = 0$.
By Hard Lefschetz theorem for cohomologies, one has $y_0=0$. This gives
\begin{equation*}
\begin{aligned}
f^*(L^{n-2p} x) + & f^*L^{n-2p}  j_*(\sum_{i=1}^r h^{p-1-i} g^*y_i)+
\sum_{l=1}^{n-2p} \binom{n-2p}{l} m^l f^*(L^{n-2p-l} x) E^l
\\ & + j_*(\sum_{i=1}^r h^{p-1-i} g^* y_i) \sum_{l=1}^{n-2p}
\binom{n-2p}{l} m^l f^*(L^{n-2p-l}) E^l=0.
\end{aligned}
\end{equation*}

Note that $l \leq n-2p \leq r-1$, we have $f_*(E^l)=0$. As
$l+p-1-i \leq n-p-2 \leq n-d-2=r-1,$ we obtain $f_*(j_*(h^{p-1-i}
g^* y_i) E^l) = (-1)^l f_* j_* (h^{l+p-1-i} g^* y_i)=0$. Apply
$f_*$ to the above equation, then we obtain $L^{n-2p} x = 0$. By
hypothesis, this implies that $x=0$. Now the above equation
is equivalent to the following (use the projection formula and
$j^* E= -h$):
\begin{equation*}
j_*(\sum_{i=1}^r h^{p-1-i} g^* (y_i \iota^* L^{n-2p}) + \sum_{i=1}^r
\sum_{l=1}^{n-2p}\binom{n-2p}{l} (-m)^l h^{l+p-1-i} g^* (y_i
\iota^*L^{n-2p-l}))=0
\end{equation*}

 Applying successively  $g_*(h^{p-d+k}j^*(\cdot))$ ($k=0, 1, \cdots$)
to the above equation and using the relations $j^* j_* = -h$ and
$g_*(h^i g^*(y)) = \delta_{i, r} y$ for $i \leq r$ and  $y \in
CH(Y)$, we obtain that $y_1=y_2=\cdots=y_d=0$.
\end{proof}

\begin{Cor}
Assume Conj. 1 (resp. Conj. 2) holds for $(X, D)$. Let $f: Z \to
X$ be the blow-up of $X$ along a smooth center of dimension $\leq
2$ (resp. $\leq 1$) and $m<0$ such that $D':=f^* D+ m Exc(f)$ is
ample. Then Conj. 1 (resp. Conj. 2) holds for $(Z, D')$.
\end{Cor}
\begin{Rque}
When $f: Z \to X$ is the blow-up of a point, then one may use
the Nakai-Moishezon criteria of ampleness to show that if
$f^* D + m Exc(f)$ is ample on $Z$, then $D$ is ample on $X$.
In this case, Conj. 1 (resp. Conj. 2) holds for $X$ implies that
it holds also for $Z$.
\end{Rque}

The following proposition can be proved in a similar way.
\begin{Prop} \label{projbundle}
(1) Assume Conjecture \ref{Conj1} (resp.  Conjecture \ref{Conj2})
holds for $X$, then it also holds for the product $X \times
\pit^m$.

(2) Conjecture \ref{Conj1} (resp.  Conjecture \ref{Conj2}) holds
for projective bundles over a smooth projective variety of
dimension $\leq 2$ (resp. $\leq 1$).
\end{Prop}

\section{Beauville's conjecture}
Let $A$ be a $g$-dimensional complex abelian variety.
 By
\cite{B2}, we have a decomposition of the Chow groups  of $A$ as
follows: $CH^p(A) = \oplus_{s=p-g}^p CH^p_{(s)}(A)$, where
$CH^p_{(s)}(A)$ consists of classes $\alpha \in CH^p(A)$ such that
for any $k \in \zit$, we have $\k^* \alpha = k^{2p-s} \alpha$,
where $\k: A \to A$ is the multiplication by $k$. It was
conjectured in \cite{B1} and \cite{B2} that $CH^p_{(s)}(A)=0$ if
$s <0$, which has been proved for $p \in \{0, 1, g-2, g-1, g\}$
(\cite{B1}). Since then, there has appeared several equivalent
formulations of Beauville's conjecture, for example, in \cite{K},
it was shown that Beauville's conjecture is equivalent to the
hypercube conjecture for abelian varieties.
 Unfortunately, despite its importance, very little progress has been made
 (neither for these reformulations) during the last twenty years.

In \cite{KV}, Kimura and Vistoli proved that Beauville's
conjecture is equivalent to the strong stability conjecture on
Chow groups of symmetric products of curves.
 To state it, we need several notations.

Let $C$ be a smooth projective connected curve and $c_0 \in C$ a
fixed point.  The $n$-th symmetric product of $C$ will be denoted
by $C^{(n)}$. Let $\phi_n: C^{(n-1)} \to C^{(n)}$ be the addition
of the point $c_0$. We will denote by $z_n \in CH^1(C^{(n)})$ the
class of the divisor $\phi_{n}(C^{(n-1)})$ and we set $z:=
z_{2g-1}$. The Jacobian of $C$ will be denoted by $J(C)$.
The strong stability conjecture of \cite{KV} asserts that for all
$n \geq 2p+1$, the map $\phi^*_n: CH^p(C^{(n)}) \to CH^p(C^{(n-1)})$
is an isomorphism;

In \cite{P}, Paranjape conjectured (which goes back to \cite{H})
the following analogue to the weak Lefschetz theorem: if $\iota: Y
\hookrightarrow X$ is the inclusion of a smooth ample divisor in a
smooth projective variety, then $\iota^*: CH^p(X) \to CH^p(Y)$ is
an isomorphism for $p <\dim(Y)/2$.  In relation with our
conjectures here, we have
\begin{Lem}
If Conjecture \ref{Conj1} holds for $X$, then for any smooth ample
divisor  $\iota: Y \hookrightarrow X$, the map $\iota^*: CH^p(X) \to CH^p(Y)$
 is injective for $p  \leq \dim(Y)/2$.
\end{Lem}
\begin{proof}
It suffices to notice that $\iota_*(\iota^* \bullet) = \bullet
\cdot Y: CH^p(X) \to CH^{p+1}(X)$, which is injective if $\dim(X)
- 2p \geq 1.$
\end{proof}
\begin{Prop}
If Conjecture \ref{Conj1} holds for all symmetric products $(C^{(n)}, z_n)$
of a curve $C$, then the strong stability conjecture holds for $C$.
\end{Prop}
\begin{proof}
It was shown in \cite{C1} that $\phi_n^*$ is always surjective,
thus we just need to prove it is injective. This follows from the
previous lemma, since
 $z_n$ is an ample divisor.
\end{proof}

It turns out that the same conclusion holds  if one  only assumes
Conjecture \ref{Conj1} for $(C^{(2g-1)},z)$, as shown by the
following:
\begin{Prop}\label{equiv}
If Conjecture \ref{Conj1} holds for $(C^{(2g-1)}, z)$, then the
strong stability conjecture holds for $C$. Furthermore
Conjecture \ref{Conj1} for $(C^{(2g-1)}, z)$ is equivalent to
Beauville's conjecture for $J(C)$.
\end{Prop}
\begin{proof}
 By Prop. 2.9 (a) \cite{KV}, the strong stability
conjecture is true if $n > 2g-1$. We may assume in the following
that $2g -1 \geq n$.  Let $i_n: C^{(n)} \to C^{(2g-1)}$ be the
composition of additions of the point $c_0$. By \cite{C1}, the
morphism $(i_n)_*: CH(C^{(n)}) \to CH(C^{(2g-1)})$ is injective
and $(i_n)_*(z_n^k) =z^{2g-1-n+k}.$

Note that we just need to prove the injectivity of $\phi_n^*$.
Let $\alpha \in CH^p(C^{(n)})$, then $$\begin{aligned} \phi_n^*
(\alpha)=0 & \Leftrightarrow (i_{n-1})_* (\phi_n^* \alpha) = 0
\Leftrightarrow (i_n)_*(\phi_n)_* (\phi_n^* \alpha)=0 \\
&\Leftrightarrow (i_n)_*(\alpha \cdot (\phi_n)_* [C^{(n-1)}])
=(i_n)_*(\alpha z_n)=0.
\end{aligned}$$

Let $p_n: C^{(n)} \to J(C)$ be the morphism $D \mapsto
\0_C(D-nc_0),$ then $p_n=p_{2g-1} \circ i_n$. By \cite{C1}, we can
write $\alpha = \sum_j p_n^*(a_j)z_n^{p-j}$ with $a_j \in
CH^j(J(C)).$ We have
$$
\begin{aligned}
(i_n)_* (\alpha z_n) &= \sum_j (i_n)_*(i_n^* p_{2g-1}^*(a_j)\cdot
z_n^{p-j+1}) = \sum_j p_{2g-1}^*(a_j) \cdot (i_n)_* (z_n^{p-j+1}) \\
&= \sum_j p_{2g-1}^*(a_j) \cdot z^{2g-n+p-j} = (\sum_j
p_{2g-1}^*(a_j)z^{p-j}) z^{2g-n}.
\end{aligned}
$$

If $\phi_n^* \alpha = 0$, then $(\sum_j p_{2g-1}^*(a_j)z^{p-j})
z^{2g-n} = 0$. Note that $z$ is ample and $2g-1-2p \geq 2g-n$. We
have $(\sum_j p_{2g-1}^*(a_j)z^{p-j}) z^{2g-1-2p} = 0$. As
Conjecture \ref{Conj1} holds for $(C^{(2g-1)},z)$, we get $\sum_j
p_{2g-1}^*(a_j)z^{p-j} =0$, which gives $\alpha = i_n^* (\sum_j
p_{2g-1}^*(a_j)z^{p-j}) = 0$.

The second statement follows from Prop. 2.18 in \cite{KV}.
\end{proof}

Now we consider Conjecture \ref{Conj1} for $(C^{(2g-1)}, z)$.
Recall that the natural map $p_k: C^{(k)} \to J(C)$ is birational
onto its image when $k \leq g$. In this case, we set $w_k =
p_{g-k}(C^{(g-k)}) \in CH^k(J(C))$ and $v_k=(-1)^k (T_c)_* (-1)_*
w_k$, where $T_c$ is the translation and $c$ is the image of the
canonical divisor on $C$ in $J(C)$. Recall that the map
$C^{(2g-1)} \to J(C)$ is a projective bundle $\pit(F) \to J(C)$
and the total Chern class of $F$ was computed in \cite{M}: $c(F) =
\sum_{k=0}^g v_k.$ In particular, the Chow group $CH(C^{(2g-1)})$
is given by $CH(J(C))[z]$ (here we identify $CH(J(C))$ with its
image under $p^*_{2g-1}$), and the minimal equation of $z$ is
$\alpha:= \sum_{k=0}^g v_k z^{g-k} = 0.$

  Let
$p$ be a non-negative integer such that $2g-1 \geq 2p+1$ and
$k:=g-p-1$. Assume first $k \leq p$. Let $y = \sum_{i=0}^p y_i
z^{p-i} \in CH^p(C^{(2g-1)})$ with $y_i \in CH^i(J(C))$ such that
$y \cdot z^{2g-2p-1} = 0$. Then by the Hard Lefschetz theorem,
we have $y_0 = 0$ in $CH^0(J(C))$.

Let $a_1= y_1, \cdots, a_k= y_k-\sum_{j=1}^{k-1} a_jv_{k-j}$. As
$y \cdot z^{2g-2p-1} = 0$ is equivalent to $\sum_{i=1}^p y_i
z^{g+k-i}=0$, by replacing powers of $z$ of degrees no less than
$g$ by using the relation $z^g = -(\sum_{i=1}^g v_i z^{g-i})$, we
obtain a relation between $1, z, \cdots, z^{g-1}$ with coefficients in
$CH(J(C))$. Thus the coefficients are zeros, which gives the
following equations in $CH(J(C))$:
$$
(1) \left\{ \begin{aligned}
 a_1 v_p + &\cdots + a_k
v_{p+1-k}  = 0 \\
& \vdots \\
a_1 v_{g-1} + &\cdots + a_k v_{g-k}  = 0
\end{aligned} \right.
$$
$$
(2) \left\{ \begin{aligned}
 y_{k+1}  =  a_1 v_k + &\cdots + a_k
v_{1} \\
& \vdots \\
y_p= a_1 v_{p-1} + &\cdots + a_k v_{p-k}
\end{aligned} \right.
$$
\begin{Lem} \label{ind}
 Let $C$ be a smooth projective curve.
Assume that the strong stability conjecture holds for $C$ with
$p=g-1-l$, then Beauville's conjecture holds for $CH^l(J(C))$.
\end{Lem}
\begin{proof}
As in \cite{KV}, we will denote by $CH^p(S^\infty C)$ the inverse
limit $\underleftarrow{\lim}  CH^p(C^{(m)})$ and $\0(1)$ the
natural line bundle on $S^\infty C$. Then by Prop. 2.9 in {\em
loc. cit.}, we have
$$
CH^p(S^\infty C)  = CH^p(J(C)) \oplus CH^{p-1}(J(C)) \cdot
c_1(\0(1)) \oplus \cdots \oplus CH^0(J) \cdot c_1(\0(1))^p.$$

 As the strong stability conjecture holds for $p=g-1-l$,
 the eigenvalues of the multiplication by $N$ on $CH^p(S^\infty C)$ are $N^0, \cdots,
N^{2(g-1-l)}$. As a consequence, for $q \leq p$, the eigenvalues
of $CH^q(J(C))$ are $N^q, \cdots, N^{q+g-1-l}.$ Thus we obtain
that $CH^q_{(q-g+l)}(J(C)) = 0$ for $q \leq g-l-1$. By the
isomorphism $CH^q_{(q-g+l)}(J(C)) \simeq CH^l_{(q-g+l)}(J(C))$ for all
$q \leq g-l-1$, we obtain that $CH^l_{(s)}(J(C)) = 0$ for all $s <0$,
i.e. Beauville's conjecture holds for $CH^l(J(C))$.
\end{proof}

\begin{Lem} \label{pbig}
Let $\Theta$ be a symmetric theta divisor on a Jacobian $J(C)$.
Assume Conjecture \ref{Conj1} for $(J(C), \Theta)$. Then
Conjecture \ref{Conj1} holds for $(C^{(2g-1)}, z, p)$ such that
$2p+1 \geq g$.
\end{Lem}
\begin{proof}
Let $k:=g-p-1$. As $2p+1 \geq g$, we have $k\leq p$ and we will
use the same notations as in the discussions before Lemma \ref{ind}.
 Let $b_i= (-1)^* T_c^* a_i, 1 \leq i \leq k$, then by
using the projection formula, equations (1) become
$$
(3) \left\{ \begin{aligned}
  b_1 w_p - b_2 w_{p-1} + &\cdots + (-1)^k b_k
w_{p+1-k}  = 0 \\
& \vdots \\
 b_1 w_{g-1} -b_2 w_{g-2} + &\cdots + (-1)^k b_k w_{g-k}  = 0
\end{aligned} \right.
$$

We  need to prove that $b_i = 0$ for any $i$ . Assume this is not
the case. Let $(b_i)_{(s)}$ be the component of $b_i$ in
$CH^i_{(s)}(J(C))$. Let $j = \min \{s \ | \ \exists i \ \text{s.t.
} (b_i)_{(s)} \neq 0 \}.$ Note that $w_i = C^{*(g-i)}/(g-i)!$,
where by abusing the notations, $C$ denotes also its image in
$J(C)$, thus $(w_i)_{(s)} = 0 $ for any $s<0$. This gives the
following equations:
$$
(4) \left\{ \begin{aligned}
  (b_1)_{(j)} (w_p)_{(0)} - (b_2)_{(j)} (w_{p-1})_{(0)} + &\cdots + (-1)^k (b_k)_{(j)}
(w_{p+1-k})_{(0)}  = 0 \\
& \vdots \\
 (b_1)_{(j)} (w_{g-1})_{(0)} -(b_2)_{(j)} (w_{g-2})_{(0)} + &\cdots + (-1)^k (b_k)_{(j)} (w_{g-k})_{(0)}  = 0
\end{aligned} \right.
$$
By Poincar\'e's formula, $C_{(0)}$  has the same cohomological
class as $\frac{\Theta^{g-1}}{(g-1)!}$ and the later lies also in
$CH^{g-1}_{(0)}(J(C))$. By \cite{B2}(p. 650), the map
$CH^{g-1}_{(0)} \to H^{2g-2}$ is injective, thus we obtain
$C_{(0)} = \frac{\Theta^{g-1}}{(g-1)!}$. By Cor. 2 \cite{B1} (p.
249), we have $(g-r)! C_{(0)}^{*r} = r! \Theta^{g-r},$ which gives
that $(w_{g-r})_{(0)} =  \frac{\Theta^{g-r}}{(g-r)!}$ for all $1
\leq r \leq g$.  Using these formulae, and noting that $p=g-k-1$,
equations (4) give the following
$$
(5) \left\{ \begin{aligned}
  (b_1)_{(j)} \Theta^{g-k-1}/(g-k-1)! &- \cdots + (-1)^{k-1} (b_k)_{(j)}
\Theta^{g-2k}/(g-2k)!  = 0 \\
& \vdots \\
 (b_1)_{(j)} \Theta^{g-2}/(g-2)! - \cdots + &(-1)^{k-1} (b_k)_{(j)}
\Theta^{g-k-1}/(g-k-1)!  = 0
\end{aligned} \right.
$$
Multiplying the j-th equation in (5) by  $\Theta^{k-j}$, we obtain
a linear system of equations of $(b_1)_{(j)} \Theta^{g-2}, \cdots,
(b_k)_{(j)} \Theta^{g-k-1}$, whose coefficient matrix has non-zero
determinant, thus  $(b_1)_{(j)} \Theta^{g-2} =0$ in $CH(J(C))$,
which implies that $(b_1)_{(j)}=0$ in $CH^1(J(C))$. We can now use
the first $(k-1)$ equations and a similar argument to deduce that
$(b_2)_{(j)} \cdot \Theta^{g-4} = 0$, which by our hypothesis
implies that $ (b_2)_{(j)} = 0$. We can continue to obtain that
$(b_i)_{(j)}=0$ for all $i$, a contradiction to the definition of
$j$.
\end{proof}
\begin{Rque} \upshape
(i) In general, one expects that if Conjecture \ref{Conj1} holds for a
smooth projective variety $X$, then it holds  for projective
bundles on $X$. Unfortunately, we can only prove this for a few
cases (see Proposition \ref{projbundle} and the previous Lemma).

(ii) With some efforts, one can prove in a similar way that the
Lemma holds without the restriction on $p$.
\end{Rque}

\begin{Thm} \label{main}
Let $J(C)$ be the Jacobian of a smooth projective curve $C$
and $\Theta$ a symmetric theta divisor.
The following conjectures are equivalent:

(i) Beauville's conjecture for $J(C)$;

(ii) Conjecture \ref{Conj1} for $J(C)$ for any ample divisor $D$;

(iii) Conjecture  \ref{Conj1} for $(J(C), \Theta)$;

(iv) Conjecture \ref{Conj1} for all $(C^{(k)}, z_k), k\geq 1$;

(v) Conjecture \ref{Conj1} for  $(C^{(2g-1)}, z)$.
\end{Thm}
\begin{proof}
The equivalence between (i) and (v) is given by Prop. \ref{equiv},
whose proof also gives the implication (v)$\Rightarrow$ (iv).

We now prove the implication (i) $\Rightarrow$ (ii). Let $A$ be an
abelian variety of dimension $g$ and $D$ an ample divisor. We
denote by $\sigma: A \to A$ the multiplication by $-1$. Then
$D=D_0+D_1$ with $D_0=(D+\sigma^* D)/2 \in CH^1_{(0)}(A)$ and $D_1
= (D-\sigma^* D)/2 \in CH^1_{(1)}(A)$. One notices that $D_0$ is
an ample symmetric divisor on $A$. By the motivic Hard Lefschetz
theorem of K\"unnemann \cite{K}, the following map $$ \cdot
(D_0)^{g+s-2p}: CH^p_{(s)}(A) \to CH^{g+s-p}_{(s)}(A), 0 \leq 2p-s
\leq g
$$
is an isomorphism. Assume $D^{g-2p} \alpha=0$ for some $\alpha \in
CH^p(A)$. By (i), we can write $\alpha = \sum_{s=0}^p
\alpha_{(s)}$ with $\alpha_{(s)} \in CH^p_{(s)}(A)$. If $\alpha
\neq 0$,  let $j=\min\{s| \alpha_{(s)} \neq 0 \}$, then the
equation $(D_0 + D_1)^{g-2p} \alpha=0$ implies that $(D_0)^{g-2p}
\alpha_{(j)} = 0$. As $j \geq 0$, we get $(D_0)^{g-2p+j}
\alpha_{(j)} = 0$, which implies $\alpha_{(j)} = 0$, contradicting
to the definition of $j$. This proves (ii).

We now prove (iii) $\Rightarrow$ (i). If $2p+1 \geq g$, then by
Lemma \ref{pbig} and Proposition \ref{equiv},  $CH^p_{(s)}(J(C)) =0$
for $s <0$.  If $2p+1 < g$, then $2(g-p-1)+1 >g$. By Lemma
\ref{pbig}, the strong stability holds for $C$ and $g-p-1$. By
Lemma \ref{ind}, Beauville's conjecture holds for $CH^p(J(C))$.
\end{proof}

As Beauville's conjecture holds for abelian varieties of dimension
$\leq 4$ by \cite{B2}, we obtain the following
\begin{Cor}\label{sym}
Let $C$ be a smooth projective curve of genus $\leq 4$ and $k$ a
natural integer. Then Conjecture \ref{Conj1} holds for $(C^{(k)},
z_k)$.
\end{Cor}

\begin{Rque}\upshape
(i).  After the first draft of this paper, A. Beauville found another proof of the
implication (iii) $\Rightarrow$ (i) in \cite{B3}.

(ii). In \cite{S}, Soul\'e proved the map in Conjecture
\ref{Conj1} is in fact an isomorphism for abelian varieties
defined over a finite field. In this case, Beauville's conjecture
can also be proved using Frobenius maps (see remark 3 in
\cite{B2}).
\end{Rque}

\section{Murre's conjectures}

In \cite{Mu}, Murre proposed the following conjectures on the
structure of Chow groups of a smooth projective variety $X$ of dimension
$n$:

(A) there exists a Chow-K\"unneth decomposition $\Delta_X=\sum_{i=0}^{2n} \pi_i$.

(B) $\pi_0, \cdots, \pi_{j-1}$ and $\pi_{2j+1}, \cdots, \pi_{2n}$ acts as zero
on $CH^j(X)$.

(C) the filtration defined by $F^v CH^j(X) =Ker(\pi_{2j}) \cap Ker(\pi_{2j-1})
 \cdots \cap Ker(\pi_{2j-v+1})$ is independent of the choice of $\pi_i$'s.

(D) $F^1 CH^j(X) = CH^j(X)_{hom}$.

It has been proved in \cite{J} that Murre's conjectures (for all
$X$) are equivalent to the Bloch-Beilinson conjecture.
Unfortunately these conjectures are very hard to
prove, and only a very few cases are known.

 Let $A$ be a $g$-dimensional complex abelian variety
and $c_0: CH^p_{(0)}(A) \to H^{2p}(A)$ the cycle map. It is known
that Murre's conjecture (A) holds for abelian varieties (see
\cite{Mu} and the references therein). In \cite{B2}, it was
conjectured that $c_0$ is always injective.  The following result
is well-known and is easily deduced from \cite{Mu}.
\begin{Lem}
Let $A$ be an abelian variety. Then Murre's conjectures (B) and (D) hold for
$A$ if and only if Beauville's conjecture holds for $A$ and the
cycle map $c_0: CH^p_{(0)}(A) \to H^{2p}(A)$ is injective for all
$p$.
\end{Lem}
\begin{Thm}
For abelian varieties, Murre's conjectures (B) and (D) are equivalent to
Conjecture \ref{Conj2}.
\end{Thm}
\begin{proof}
Assume Conjecture \ref{Conj2} for abelian varieties,
then Conjecture \ref{Conj1} also  holds,
which implies Beauville's conjecture by Theorem \ref{main}. We
proceed by induction on $p$ to prove the injectivity of $c_0$. The
case $p=0$ is trivial. Assume we have proved it for $p-1$ with $p
\leq g/2$. Let $D$ be an ample symmetric divisor on $A$. By
\cite{K}, the map $\cdot D^{g-2p+2}: CH^{p-1}_{(0)}(A) \to
CH^{g-p+1}_{(0)}(A)$ is an isomorphism. Together with the Hard
Lefschetz theorem on cohomology, this implies the injectivity of
the map $CH^{g-p+1}_{(0)}(A) \to H^{2g-2p+2}(A).$ Now for any
element $\alpha \in CH^p_{(0)} (A)_{hom}$, the cohomology class of
$\alpha D^{g-2p+1}$ vanishes, which implies that $\alpha
D^{g-2p+1}=0$ in $CH^{g-p+1}_{(0)}(A)$. By Conjecture \ref{Conj2},
this gives $\alpha=0$. In other words, the map $CH^p_{(0)}(A) \to
H^{2p}(A)$ is injective.

Assume Murre's conjecture. Then the injectivity of $c_0$ and
Beauville's conjecture  imply $CH^p(A)_{hom} = \oplus_{s \geq 1}
CH^p_{(s)}(A).$ Now a similar argument as done in the proof of
Theorem \ref{main} implies the injectivity of $\cdot D^{g-2p+1}:
\oplus_{s \geq 1} CH^p_{(s)}(A) \to  \oplus_{s \geq 1}
CH^{g-p+1}_{(s)}(A)$ for any ample divisor $D$ on $A$, concluding
the proof.
\end{proof}

In a similar way as the proof of Lemma \ref{pbig}, one can show
that if Murre's conjectures (B) and (D) hold for a Jacobian
$J(C)$, then Conjecture \ref{Conj2} holds for $(C^{(k)}, z_k)$.
This give the following
\begin{Cor}
Conjecture \ref{Conj2} holds for  $(C^{(k)}, z_k)$ for all curves
$C$ of genus $\leq 3$.
\end{Cor}

\end{document}